\documentclass[11pt,reqno]{article}
\usepackage{amsmath,amssymb,amscd,amsfonts}
\usepackage{color}
\usepackage[normalem]{ulem}
\usepackage{latexsym}
\usepackage{graphicx}
\usepackage{amsmath}
\newtheorem{thm}{Theorem}[section]

\newtheorem{lem}{Lemma}[section]

\newtheorem{fact}{Fact}[section]

\newtheorem{alg}{Procedure}[section]

\newcommand{\qed}{\hspace*{\fill} \rule{7pt}{7pt}}

\newcommand{\rw}[1]{{\color{black} #1}}
\newcommand{\jz}[1]{{\color{black} #1}}
\newcommand{\ps}[1]{{\color{black} #1}}

\begin{document}

\def\concat#1#2{\sideset{_{#1}}{_{#2}}{\mathop{\circ}}}

\title{Multicoloring of cannonball graphs}

\small{
\author{
Petra \v{S}parl\\
University of Maribor,\\
Faculty of Organizational Sciences,\\
Kidri\v{c}eva 55a, \\
SI-4000 Kranj, Slovenia\\
and \\
IMFM,
Ljubljana, Slovenia\\
{\tt petra.sparl@fov.uni-mb.si}
\and 
Rafa\l ~Witkowski\thanks{This work was supported by grant N206 1842 33 for years 2011-2014}\\
Adam Mickiewicz University,\\Faculty of Mathematics \\ and Computer Science,\\ul. Umultowska 87\\ 61-614 Pozna\'n, Poland\\{\tt rmiw@amu.edu.pl}
\and Janez \v{Z}erovnik\thanks{Supported in part by ARRS, the research agency of Slovenia.}\\
University of Ljubljana,\\
Faculty of Mechanical Engineering, \\
A\v sker\v ceva 6, \\
SI-1000 Ljubljana, Slovenia\\
and \\
IMFM, 
Ljubljana, Slovenia\\
{\tt janez.zerovnik@fs.uni-lj.si, janez.zerovnik@imfm.si }
}
\date{\today}

\maketitle

%%%%%%OLD ABSTRACT 
%\begin{abstract}
%In the frequency allocation problem, we are given a~cellular telephone network whose geographical coverage area is divided into cells, where phone calls are serviced by
% assigned frequencies, so that none of the pairs of calls emanating from the same or neighboring cells is assigned the same frequency. The problem is to use the frequencies
% efficiently, i.e.~minimize the span of frequencies used. The frequency allocation problem can be regarded as a~multicoloring problem on a~weighted hexagonal graph, 
%where each vertex knows its position in the graph. We  generalize    this problem into higher dimension and present first \jz{approximation} 
%algorithms for multicoloring of  so called cannonball graphs.
%\end{abstract}

%%%SHORTER ABSTRACT ... janez 10.7.2013
\begin{abstract}
The frequency allocation problem  that appeared in the design of cellular telephone networks can be regarded as a~multicoloring problem on a~weighted hexagonal graph, 
which opened some still interesting mathematical problems.  We  generalize    the multicoloring  problem into higher dimension and present the  first \jz{approximation} 
algorithms for multicoloring of  so called cannonball graphs.
\end{abstract}

%%%%%%%%%%%%%%%%%%%%%%%%%%%%%%%%%%%%%%%%%%%%%%%%%%%%%%%%
\section{Introduction}
%%%%%%%%%%%%%%%%%%%%%%%%%%%%%%%%%%%%%%%%%%%%%%%%%%%%%%%%

A fundamental problem that appeared  in \ps{the} design  of   cellular networks is to assign  sets of frequencies to transmitters  in order to avoid \rw{the} unacceptable interferences.
 The number of frequencies demanded at a transmitter may vary between   transmitters.  The problem appeared in the sixties and  was soon related to \rw{graph} multicoloring
  problem  (see the early survey \cite{Hale}).  It received \rw{an} enormous attention in the nineties and is still of considerable interest  (see  \cite{Annals} and the references there). 
Besides the mobile telephony there are several  applications of frequency assignment including  radio and television broadcasting, military applications, 
satellite communication and wireless LAN   \cite{Annals}. 
A sizable  part of theoretical studies is  concentrated on the simplified model when  the underlying graph which has to be multicolored is a subgraph of triangular grid   (see \cite{pierwsza, channel, Hale}).
This is a natural choice because   it is well known that   hexagonal cells provide a coverage with \rw{the} optimal ratio of  the distance between centers compared to the area covered by each cell.
Such graphs are called   hexagonal graphs \cite{a43,a54,tfree}.
Indeed, the model is a reasonable approximation for the rural cellular networks where the underlying graph is often nearly planar, 		\rw{and a} popular example 
are the   sets of benchmark problems    based on \rw{the} real  cellular  network  around  Philadelphia  \cite{Anderson} (see the FAP website  \cite{FAPpage}). 
Although the multicoloring of hexagonal graphs seems to be a very simplified optimization problem, some interesting mathematical questions were asked at the time  
that are still open. An example is the  conjecture \jz{of McDiarmid and Reed} saying  that the multichromatic number \ps{(the formal definition is given on page \pageref{hi})} 
%\raf{we didn't mentioned before what the multichromatic number is}
 of any hexagonal graph $G$  is between   
$\omega(G)$ and $9\omega(G)/8$,   where $\omega(G)$  is the weighted clique number \cite{pierwsza}.
On the other hand, the hexagonal graph  model is known to be practically  useless in urban areas, 
where high concrete  buildings on one hand prevent propagation of \rw{the} radio  signals  and 
on the other hand allow very high concentration of users.  
Loosely speaking, a three dimensional  model may be needed in contrast to the hexagonal graphs that are \ps{a} good \ps{model} for two dimensional  networks. 
In this paper  we discuss a generalization of the multicoloring problem on hexagonal graphs from planar case to  three dimensions. 
It is well known that   hexagonal cells of the same size with centers positioned in the triangular grid    provide an optimal coverage of the plane.
Optimality here means the best ratio between the diameter and the area covered by the cell.
The situation is much more interesting in three dimensions.
Obviously, optimal cells would be nearly balls, and the question is how to position \rw{the} centers of the balls to achieve \rw{the} optimal diameter to \rw{the} volume ratio.
The famous  Kepler conjecture   was  \ps{a} longstanding    conjecture about \rw{the} ball packing in three-dimensional Euclidean space. 
It says that no arrangement of  equally sized balls filling space has  greater average density than that of the cubic close packing (face-centered cubic) and \rw{the} hexagonal close packing arrangements. The density of these arrangements is slightly greater than 74\%. It may be interesting to note that the solution of Kepler's conjecture is included as a part of 18th problem in the famous   Hilbert's problem list back in 1900 \cite{hilbert}.
%
%\begin{figure}
%  \begin{center}
%    \includegraphics[width=2.5in]{kball.eps}   
%~~~ ~~~
%    \includegraphics[width=1.5in]{Kepler_conjecture_2.eps} 
%\caption{ One of the  two optimal arrangements
%and 
%a diagram from Johannes Kepler's 1611 {\em  Strena Seu de Nive Sexangula}.}
%\end{center}
%\end{figure}
%
Recently  Thomas Hales, following an approach suggested by Fejes Toth, published a proof of the Kepler conjecture.  
 For more details, see \cite{Hales2000,Hales2005i}.
Given an optimal arrangement of   balls, we define a graph by taking the balls (or   centers of  balls) as vertices and  connect \rw{each} \ps{pair of} touching balls with \rw{an} edge.
Nonnegative demands are assigned to each vertex and  we are interested in multicoloring of the  graph induced on  vertices of positive demand.
Loosely speaking, we generalize the problem of  multicoloring of hexagonal graphs from two dimensions to three dimensions. 
The question has been asked at the Oberwolfach seminar  Algorithmische Graphentheorie \cite{oberwolfach}  and we are not aware of any result since then.

More formally,  
we are interested in   multicoloring of weighted graphs $G=(V(G), E(G), d)$, where
$V=V(G)$ is the set of  vertices, $E=E(G)$ is the set of edges, and  $d$ assigns  a positive  integer $d(v)$  to   vertex $v\in V$.
$d(v)$ is the {\em weight} of a vertex, here also called {\em demand}.
Adjacent vertices are called  {\em neighbors}. 
The {\em degree} of a vertex,   $deg_G(v)  = deg(v) $ is the number of neighbors of $v$.
A {\em proper multicoloring} of $G$ is a mapping $f$ from $V(G)$ to subsets of integers such that $\left\vert f(v)\right\vert \geq d(v)$ for any vertex  $v\in V(G)$ and
 $f(v)\cap f(u)=\emptyset$ for any pair of adjacent vertices $u$ and $v$ in the graph $G$. The minimal cardinality of a proper multicoloring of $G$, $\chi_{m}(G)$, is called the {\em multichromatic number}. \label{hi} Another invariant of interest in this context is the {\em(weighted) clique number}, $\omega(G)$, defined as follows: 
The weight of a clique of $G$ is the sum of demands on its vertices and $\omega(G)$ is the maximal clique weight on $G$. 
Clearly, ${\chi}_{m}(G) \geq \omega(G)$.
 {\em Hexagonal graph} is \rw{the} graph induced on vertices of triangular grid of positive demand. Or, in other words, cells of hexagonal grid are assigned integer demands, 
and the graph is \ps{composed  } by taking cells as vertices and \rw{two} hexagons sharing an edge are regarded to be adjacent.
In 3-dimensional case we will consider   optimal arrangements of   balls, and define a graph by taking balls (with positive demand) as vertices, and connect touching balls by edges.
We call  these graphs \rw{the} {\em cannonball graphs} as Keplers motivation for studying the arrangements of balls was  optimal arrangement   of   cannonballs.
In the last decade there were several results on upper bounds for the  multichromatic number in terms of weighted clique number for  hexagonal graphs, some of which also provide   approximation algorithms that are fully distributed and run in constant time
\cite{Havet,havetJZ,Janssen,pierwsza,channel,cykle,a73lin,a43,a54,a76subclass,ost,tfree,a1712,a43moj,isco,a65}.
The best \rw{known} approximation ratios   are 
${\chi}_{m}(G)\leq(4/3)\omega(G)+O(1)$  in general   \cite{pierwsza,cykle,a43} and 
${\chi}_{m}(G)\leq(7/6)\omega(G)+O(1)$ for triangle free hexagonal graphs   \cite{Havet,ost,a73lin}.
The conjecture of McDiarmid and Reed: ${\chi}_{m}(G)\leq(9/8)\omega(G)+O(1)$ remains an open problem \cite{pierwsza}.
 
No approximation algorithm  and no upper bound was previously  known for \rw{the} multichromatic number of cannonball graphs.
Here we give two upper bounds, where the first is easily implied by known  results for hexagonal graphs  (because  a layer in a cannonball graph is a hexagonal graph)
and  the second is an improvement of the first upper bound using some  structural properties of the cannonball graphs. 
In both cases, constructions are given thus providing polynomial approximation algorithms.
The main result of this paper  that gives the first answer to the problem asked in  \cite{oberwolfach} is 

\begin{thm}\label{twierdzenie}
There is an approximation algorithm for   multicoloring cannonball graphs which uses at  most $\frac{11}{6}\omega(G) + O(1)$ colors.  
Time complexity of the algorithm is polynomial.
\end{thm}

The paper is organized as follows. In the next section we formally define some basic terminology. 
In Section \ref{algorytm}, we present an overview of the algorithm, while in Section \ref{dowody} 
we provide a proof of Theorem \ref{twierdzenie}. In the last Section we give some ideas for futher work.

%%%%%%%%%%%%%%%%%%%%%%%%%%%%%%%%%%%%%%%%%%%%%%%%%%%%%%%%
\section{Hexagonal and cannonball graphs}\label{preliminary}
%%%%%%%%%%%%%%%%%%%%%%%%%%%%%%%%%%%%%%%%%%%%%%%%%%%%%%%%

First we formally define hexagonal and cannonball graphs.
Recall the  formal  definition of hexagonal graphs: 
the position of each vertex is an integer linear combination $x\vec{p}+y\vec{q}$ of two vectors $\vec{p}=(1,0)$  and $\vec{q}=(\frac{1}{2}, \frac{\sqrt{3}}{2})$
and  the  vertices of the triangular \ps{grid  } are identified with pairs $(x,y)$ of integers. 
Put an edge connecting two vertices if the points representing the vertices are at distance one in the triangular grid (in other words, when the corresponding hexagonal cells are adjacent).
To construct a hexagonal graph $G$,  positive weights  are assigned to a finite subset of points in the grid
and $G$ is a subgraph induced on $V(G)$,  the set of  grid vertices with positive weights.   
  Cannonball graphs are constructed in a similar way. However, we have many possibilities  already when constructing   the underlying grid,  
which, loosely speaking,  consists of tetrahedrons and will be called  {\em tetrahedron grid}  $T$. 
%\raf{shouldn't we use the name: tetrahedron\rw{al} grid?} \pet{I don't think so: in Google for tetrahedron grid there are 3790 matches and for tetrahedronal only 46}
Optimal arrangement of   balls in one layer is to put the centers of balls in \rw{the} points of triangular grid. 
Then, there are exactly two possibilities to put a second layer on the top of the first \ps{layer}. 
These two arrangements are obviously symmetric, however, when choosing a position for the third layer, there are two possibilities that give rise to 
different arrangements. We will call them {\em layer-arrangement (a)} and {\em layer-arrangement (b)}, respectively (see figure \ref{BallPacking}).

\begin{figure}[h]
 \begin{center}
 \includegraphics[height=12cm]{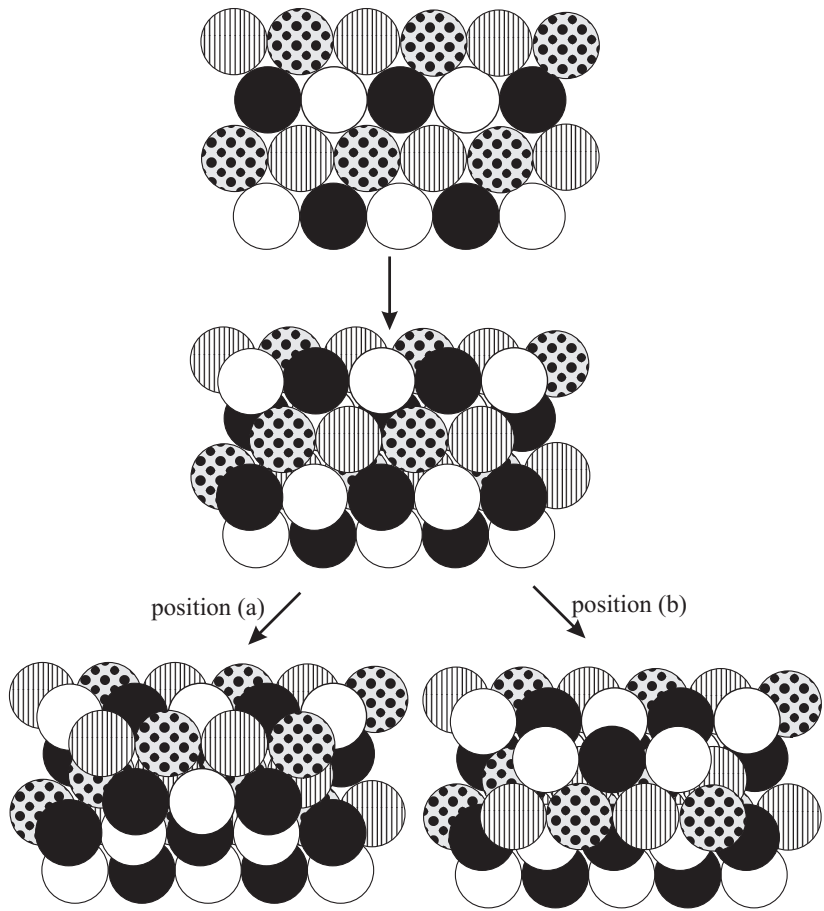}
 \caption{Two different arrangements of the third layer.} 
 \label{BallPacking}
 \end{center}
\end{figure}

Consequently, we have an infinite number of tetrahedron grids, that all came from  the optimal ball arrangements.
One of the arrangements, called the  cubic close packing (see case (a) of figure \ref{BallPacking}),
can be described nicely by introducing a third vector $\vec{r}=(\frac{1}{2}, \frac{\sqrt{3}}{6}, \frac{\sqrt{6}}{3})$
in addition to  $\vec{p}=(1,0,0)$ and $\vec{q}=(\frac{1}{2}, \frac{\sqrt{3}}{2},0)$.
Now  the position of each vertex is an integer linear combination $x\vec{p}+y\vec{q}+z\vec{r}$ 
and the vertices of the triangular \ps{grid} may be identified with  triplets $(x,y,z)$ of integers. 
Given the vertex $v$, we will refer to its coordinates as $x(v)$, $y(v)$ and $z(v)$, 
or shortly $x$, $y$, and $z$,  when there is no confusion possible.
For other  arrangements  there is no  such  easy extension of  the notation from hexagonal graphs.
A  cannonball graph $G$  is obtained by assigning  integer weights to the points of the tetrahedron grid $T$, 
taking as   $V(G)$  the vertices  in the grid with positive weights, and introducing edges between vertices  at  euclidean distance one
 (in other words, connecting the touching balls). 
The cannonball graphs based on the cubic close packing will be called  {\em regular cannonball graphs}.
Clearly, from the construction it follows that any layer of a cannonball graph is a hexagonal graph (maybe not connected).

Formally, a {\em cannonball graph} is a graph induced on vertices of positive weight.
%\rw{\sout{We will say that a cannonball graph is regular, if it is based on the cubic close packing, and is irregular otherwise. 
%Hence regular cannonball graphs allow natural definition of coordinates.}} \raf{It was before, I think}

There is a natural basic 4-coloring  of  (unweighted) cannonball graph.
Start with any layer and call it the base layer.  Introduce  coordinates  $(x,y,0)$  in this layer and define  the base coloring by the  formula 
\begin{equation}
bc(v)=x  \bmod 2 + 2 (y \bmod 2) . \label{base_colors_2D}
\end{equation}
 Colors of vertices of the next  layers   are \ps{then} determined exactly as follows.
It is obvious that whenever we store a new layer  above  (or under)  the previous one  with fixed coloring,  we know that  each ball  from the new layer  is connected to  exactly
three  balls from the previous  layer, and all of those balls   have different colors.  
Thus there is exactly one extension of the  four  coloring to the next layer (see figure\rw{s} \ref{BallPacking} and   \ref{12neigh}, where 4-coloring, using colors  $0,1,2,3$}, is presented). %\rw{\sout{Note that circles and gray lines represent the middle layer containing $v$, squares and thick lines represent the upper layer, 
%dashed triangles and dashed lines represent the lower layer}}).\raf{I move it to the figure, like with the figures in the proof} 
It is easy to see that this rule, starting from (\ref{base_colors_2D}),  gives  a proper coloring of the next layers.
In regular cannonball graphs this coloring can be  given by closed  expression  in the following way:
\begin{equation}
bc(v) = ((z+1)\bmod 2) (x \bmod 2 + 2 (y \bmod 2)) + (z \bmod 2)((x+1)\bmod 2 + 2 ((y+1) \bmod 2)) .
\end{equation}

From the construction of cannonball graphs it is clear that each vertex has (at most) 6 neighbors in its layer, 
and in addition (at most) three neighbors in each of  the neighboring layers. 
The degree of a vertex in cannonball graph is hence at most 12 (see figure \ref {12neigh}).

\begin{figure}[h]
 \begin{center}
 \includegraphics[height=4cm]{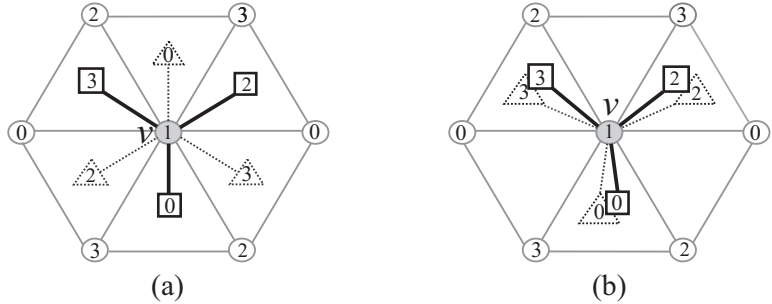}
 \caption{All possible 12 neighbors of vertex $v \in G$ for arrangements (a) and (b).
 \rw{Circles and gray lines represent the middle layer containing $v$, squares and thick lines represent the upper layer, dashed triangles and dashed lines represent the lower layer.}
 } 
 \label{12neigh}
 \end{center}
\end{figure}

The cliques in the cannonball graphs can have at most four vertices.
The \emph{(weighted) clique number}, $\omega(G)$, is the maximal clique weight on $G$, where the weight of a clique is the sum of weights on its vertices. 
 As cliques in cannonball  graphs can  have at most four vertices,
 the weighted clique number is the maximum weight over weights of all tetrahedrons, triangles, edges and weights of isolated vertices.
  Therefore,  we can  define  invariants  $\omega_i(G)$ which  denote  the maximal weight of clique of size at most $i$ on $G$. 
In fact, we can regard the clique numbers as based on the complete subgraphs of the grid graph because the vertices of weight 0 clearly do not contribute to the clique weights.
  For example,  $\omega_2(G)$ is the maximal weight over all edges and isolated vertices.
Clearly, for cannonball  graphs we have 
$$ \omega_1(G) \leq \omega_2(G) \leq \omega_3(G) \leq \omega_4(G) = \omega(G).
$$
 
An induced subgraph of the cannonball graph  without 3-clique will be  called a~{\em triangle-free cannonball graph}.

In the algorithm we will consider some subgraphs of the cannonball graph, in particular, it may  be useful to have   bipartite \ps{subgraphs} \jz{ and}  3-colorable subgraph\jz{s}.

%For bipartite graphs we know that the multichromatic number equals the weighted clique number. 
In \cite{pierwsza} it was proved that for any weighted bipartite graph $H$, $\chi_m(H) = \omega(H)$. Bipartite graph $H$  can be optimally multicolored by the following procedure:

\begin{alg}\label{dwukolorowanie}\cite{algorithmica}
Let $H = (V', V'', E, d)$ be a~weighted bipartite graph. We get an optimal multicoloring of $H$ if
to each vertex $v \in V'$ we assign a~set of colors $\{1, 2, \ldots, d(v)\}$, while with each vertex $v \in V''$ we associate a~set of colors  
$\{m(v)+1, m(v)+2,\dots , m(v)+d(v)\}$, where $m(v)=\max\{d(u): \{u,v\} \in E\}$.
\end{alg}

For 3-colorable graphs, there is a simple $\frac{3}{2}$-approximation coloring algorithm.
\begin{lem}\label{kgoodlemma}\cite{tfree}
Every 3-colorable graph can be multicolored using at most $\frac{3}{2}\omega(G)+O(1)$ colors.
\end{lem}

Using the proof of Lemma \ref{kgoodlemma} from \cite{tfree}, we can give a procedure for $\frac{3}{2}\omega(G)$-coloring of any $3$-colorable graph in the following way:

\begin{alg}\label{trzykolorowanie}
Let $H = (V,E,d)$ be a weighted, 3-colorable graph, \ps{colored} with colors $C,M,Y$. Let $V = V_C \cup V_M \cup V_Y$ be \rw{the} sets of vertices with colors $C,M,Y$ respectively.
Construct  three new weighted graph\rw{s} $H_i = (V_i,E_i,d_i)$, where \ps{$i \in\left\{  1,2,3 \right\}$}, $d_i(v) = d(v)/2$ \ps{for every $i \in V_{i}$}, $V_i = V \backslash V_j$ 
\ps{for} $j \in \{C,M,Y\}$ respectively, and $E_i \subseteq E$ is the set of all edges in $H$ with both endpoints in  $V_i$ ($H_i$ is induced by $V_i$).
 Each $H_i$ is bipartite since we have a 2-coloring of this graph. Use Procedure \ref{dwukolorowanie} to optimally multicolor graphs $H_1$, $H_2$ and $H_3$. 
Combining all these colorings we get a $\frac{3}{2}\omega(G)$-coloring.
\end{alg}

Recall that by definition all vertices  of \ps{a tetrahedron} grid which are not in $G$ must have weight $d(v)=0$.
 \jz{Then we need not  check whether \ps{a} vertex \ps{of the grid} is one of the vertices of $G$.}
Therefore:
$$ \omega_3(G) = \max\{d(u)+d(v)+d(t): \{u,v,t\} \in \tau(T) \},$$ where $\tau(T)$ is the set of all triangles of \ps{a} tetrahedron grid $T$.

\begin{def}\label{baseFunction}
For each vertex $v\in G$, define {\em base function} $\kappa$ as
$$\kappa (v) =  \max \{ a(v,u,t): \left\{v,u,t\right\}  \in \tau(T) \} ,$$
where
$$a(u,v,t) = \left\lceil \frac{d(u)+d(v)+d(t)}{3} \right\rceil,$$
is an average weight of the triangle $\left\{u,v,t\right\}\in~\tau(T)$.
\end{def}

Clearly, the following fact holds.

\begin{fact}\label{komega}
For each $v \in G$,  $$ \kappa (v) \leq \left\lceil \frac{\omega_3(G)}{3} \right\rceil \leq \left\lceil \frac{\omega(G)}{3} \right\rceil $$
\end{fact}

We call \rw{a} vertex $v$ {\em heavy} if $d(v)>\kappa(v)$, otherwise we call it {\em light}. If $d(v)>2\kappa(v)$, we say that the vertex $v$ is {\em very heavy}.

To color vertices of $G$ we use colors from an appropriate {\em palette}. For a given color $c$, its palette is defined as a~set of pairs $\{(c,i)\}_{i\in \mathbb{N} }$. A palette is called a {\em base color palette} if $c\in \{0,1,2,3\}$ is one of the base colors, and it is called \rw{an} {\em additional color palette} if $c \notin \{0,1,2,3\}$. 

If a vertex $v$ does not have a neighbor of color $i$ in $G$, we call such color a {\em free color} of $v$.

%%%%%%%%%%%%%%%%%%%%%%%%%%%%%%%%%%%%%%%%%%%%%%%%%%%%%%%%
\section{Algorithms for multicoloring cannonball graphs}\label{algorytm}
%%%%%%%%%%%%%%%%%%%%%%%%%%%%%%%%%%%%%%%%%%%%%%%%%%%%%%%%

%\subsection*{Trivial $\frac{8}{3} \omega(G)$-aproximation algorithm}\label{easy}

Recall that a tetrahedron grid consists of   several horizontal layers which are triangular grids.  
No matter how we store one layer onto another, for every hexagonal graph in a particular horizontal layer one of the well known algorithms \cite{pierwsza, a43, a3324}
may be used. The best known approximation ratio is $\frac{4}{3}\omega(G')$, where $G'$ is \ps{a} hexagonal graph in a single layer (obviously $\omega(G') \leq \omega(G)$). 
Therefore, for each layer we need at most $\frac{4}{3}\omega(G)$ colors. We can use one palette of colors for odd layers and the second palette of colors for even layers, in order to prevent any conflict. All together we get an algorithm that uses at most $2 \cdot \frac{4}{3}\omega(G) = \frac{8}{3}\omega(G)$.
Since this bound is obviously not the best possible, the algorithm that improves this bound is presented in the continuation.

In many papers, e.g. \cite{pierwsza, a43, a1712, isco} a strategy of borrowing was used. The same idea can be used for  cannonball graphs. 
Our algorithm consists of two main phases. In the first phase
(Steps~1 and~2 of the algorithm) vertices take $\kappa(v)$ colors from its base color palette,
so use no more than   $\frac{4}{3}\omega(G)$  colors. After this phase, all light vertices in $G$ are fully colored, i.e. every light vertex
$v \in V(G)$ already  received all needed $d(v)$ colors.
The vertices that are heavy, but not very heavy, induce a~triangle-free cannonball graph with \rw{the} weighted clique number not exceeding
$\left\lceil \omega(G)/3 \right\rceil$.
Very heavy vertices in $G$ are isolated in the remaining graph and therefore they can easily be fully colored 
  (Step~2).
In the second phase (Steps~3 and~4 of the algorithm) we first color all vertices of degree 4 and \ps{thereby obtain a  3-colorable
 graph, for which  } Procedure \ref{trzykolorowanie} \ps{can be used} for satisfying the remaining demands   by using new colors.

More precisely, our algorithm consists of the following steps:

\subsection*{Algorithm}
\begin{description}

\item[{\bf Input:}] \ps{A} weighted cannonball graph $G=(V,E,d)$. Coordinates $(x(v),y(v),z(v))$, for $v\in V$.
%, where all vertices know its position in the graph.
\item[{\bf Output:}] A proper multicoloring of $G$, using at most $\frac{11}{6}\cdot\omega \left( G \right) + O(1)$ colors.

\item[{\bf Step 0}] For each vertex $v \in V$ compute its base color $bc(v)$
%$$bc(v) = ((z+1)\mod 2) (x \mod 2 + 2 (y \mod 2)) + (z \mod 2)((x+1)\mod 2 + 2 ((y+1) \mod 2)) .$$
and its base function value
$$\kappa (v) = \max \left\{\left\lceil\frac{d(u)+d(v)+d(t)}{3} \right\rceil: \{v,u,t\}  \in \tau(T) \right\},$$
where $\tau(T)$ is \rw{the} set of all triangles in the tetrahedron grid $T$.

\item[{\bf Step 1}] For each vertex $v \in V$ assign $\min\{\kappa (v), d(v)\}$ colors from its base color palette to $v$. 
Construct a~new weighted triangle-free cannonball graph $G_1 = (V_1,E_1,d_1)$ where $d_1(v) = \max\{d(v)-\kappa (v), 0\}$, $V_1 \subseteq V$ is 
the set of vertices with $d_1(v)>0$ (heavy vertices in $G$) and $E_1 \subseteq E$ is the set of all edges in $G$ with both endpoints \ps{in} $V_1$ ($G_1$ is induced by $V_1$).

\item[{\bf Step 2}] For each vertex $v \in V_1$ with $d_1(v)>\kappa (v)$ (very heavy vertices in $G$) assign the first unused  $\kappa(v)$ colors of the base color 
palettes of its neighbors in the tetrahedron grid $T$. Construct a~new graph $G_2 = (V_2, E_2, d_2)$ where $d_2\left(v\right)$ is the difference
 between $d_1(v)$ and the number of colors assigned  in this step, $V_2 \subseteq V_1$ is the set of vertices with $d_2(v)>0$ and $E_2 \subseteq E_1$ is the set of 
all edges in $G_1$ with both endpoints \ps{in} $V_2$ ($G_2$ is induced by $V_2$).

\item[{\bf Step 3}] For each vertex $v \in V_2$ with \ps{$deg_{G_2}(v)=4$} assign unused colors from \ps{its} free \ps{base} color palette. 
Construct a~new 3-colorable graph $G_3 = (V_3, E_3, d_3)$ where $d_3\left(v\right)$ is the difference between  $d_2(v)$ and the number of colors assigned  in this step, 
$V_3 \subseteq V_2$ is the set of vertices with $d_3(v)>0$ and $E_3 \subseteq E_2$ is the set of all edges in $G_2$ with both endpoints \ps{in} $V_3$ ($G_3$
 is induced by $V_3$).

\item[{\bf Step 4}]  Apply Procedure \ref{trzykolorowanie} for \rw{the} graph $G_3$ by using colors from new additional color palettes.

\end{description}

%%%%%%%%%%%%%%%%%%%%%%%%%%%%%%%%%%%%%%%%%%%%%%%%%%%%%%%%
\section{Correctness proof}
%%%%%%%%%%%%%%%%%%%%%%%%%%%%%%%%%%%%%%%%%%%%%%%%%%%%%%%%
\label{dowody}

\rw{Recall that each vertex knows its position on the tetrahedron \ps{grid} $T$.}
Note that whenever we mention "very heavy/heavy/light vertex", we refer to the property of
this vertex in \ps{a} graph $G$, i.e. there is no reclassification in graphs $G_i, i\in{1,2,3}$.
\\ \\
In Step 0 we have to prove that each vertex can obtain its base color. Recall that we can assume that one of the horizontal layers is the base layer.
 We can compute the base coloring in each vertex $v=(x,y)$ of this layer by formula: $bc(v)=x \mod 2 + 2 (y \mod 2)$. 
\ps{ In} the neighboring layers (the above and the bottom one)   the colors \ps{are} determined by 4-coloring of the base layer. Thus, we can obtain a proper 4-coloring for the whole cannonball graph.
\\ \\
In Step 1 each heavy vertex $v$ in $G$ is assigned $\kappa (v)$ colors from its base color palette, while each light vertex $u$ is assigned $d(u)$ colors from its base color palette.
Hence the remaining weight of each vertex $v \in G_1$ is
$$d_1(v) = d(v)-\kappa(v).$$
Note that $G_1$ consists only of heavy vertices in $G$. Therefore

\begin{lem}\label{G1tf}
$G_1$ is a~triangle-free cannonball graph.
\end{lem}
%and the proof of that can be found in \cite{a139}.
\begin{proof}
Assume that there exists a triangle $\{v,u,t\} \in G_1$, which means that $d_1(v),d_1(u),d_1(t)>0$. Then we have:
\begin{eqnarray}
d(v) + d(u) + d(t) &=& d_1(v)+\kappa(v) + d_1(u) + \kappa(u) + d_1(t) + \kappa(t)    \geq \nonumber\\
                &\geq& d_1(v) + d_1(u) + d_1(t) + 3a(u,v,t)   \geq \nonumber\\
                &\geq& d_1(v) + d_1(u) + d_1(t) + d(v) + d(u) + d(t) \nonumber\\
                &>& d(v) + d(u) + d(t) \nonumber
\end{eqnarray}
a contradiction. Therefore, the graph $G_1$ does not contain a 3-clique, so it is a~triangle-free cannonball graph.
\qed
\end{proof}

In Step 2 only vertices with $d_1(v)>\kappa (v)$ (very heavy vertices in $G$) are colored. Each
very heavy vertex in $G$ has enough unused colors in its neighborhood to be finally multicolored.
Namely, if  a  vertex $v$ is very heavy in $G$ then it is isolated in $G_1$ (all its neighbors are light in $G$). 
% (see \cite{a75}). 
Otherwise, for some $\{v,u,t\}\in\tau(T)$ we would have $$d(v)+d(u) > 2\kappa (v)+\kappa(u) \ge 3a(v,u,t) \ge d(v)+d(u),$$
a~contradiction.
Without loss of generality we may assume that $bc(v) = 0$.  Denote  
$$ D_1(v) = \min\{\kappa (v) - d(u): \left\{u,v\right\} \in T, bc(u)=1\}, \label{D1def}$$ 
$$ D_2(v) = \min\{\kappa (v) - d(u): \left\{u,v\right\} \in T, bc(u)=2\}, $$
$$ D_3(v) = \min\{\kappa (v) - d(u): \left\{u,v\right\} \in T, bc(u)=3\}. $$
Obviously, $D_1(v),D_2(v),D_3(v)>0$ for very heavy vertices $v$ in $G$. Since in Step 1 each light vertex $t$ uses exactly $d(t)$ colors from its base color palette, we have at least $D_i(v)$ free colors from the $i$-th base color palette. 
%Let us show that if \jz{the} vertex $v$ could assign those colors to itself, we would have $G_2$ with $\omega(G_2) \leq \lceil \omega(G)/3 \rceil$. 
Formally it can be proved that

\begin{lem}\label{nierownoscG1}
In $G_1$ for every edge $\left\{v,u\right\} \in E_1$ \jz{ we have}:
$$ d_1(v) + d_1(u) \leq \kappa(v), \quad d_1(u) + d_1(v) \leq \kappa(u).$$
\end{lem}
%and the proof can be found in \cite{a1712}.
\begin{proof}
Assume that $v$ and $u$ are heavy vertices in $G$ and $d_1(v) + d_1(u) > \kappa (v)$.
Then for some $\{v,u,t\}\in\tau(T)$ we have:
$$ d(v)+d(u) = d_1(v) + \kappa (v) + d_1(u) + \kappa(u) > 2\kappa (v) + \kappa(u) \geq 3a(u,v,t) \geq d(u) + d(v),$$ 
a~contradiction.
\qed
\end{proof}

%From these \rw{\sout{facts }lemmas} we can observe that
Further useful observation is 

\begin{fact}\label{Stage1}
$$\omega(G_2) \leq \left\lceil \frac{\omega(G)}{3} \right\rceil .$$
\end{fact}
\begin{proof}
Recall that in a~cannonball graph the only cliques are tetrahedrons, triangles, edges and isolated vertices.
Since $G_1$ is a~triangle-free cannonball graph, $G_2$ contains no tetrahedron, neither triangle, so we have only edges and isolated vertices to check.

For each edge $ vu \in E_2$, using Lemma \ref{nierownoscG1} and Fact \ref{komega}, we have:
$$ d_2(v) + d_2(u) \leq d_1(v) + d_1(u) \leq \kappa (v) \leq \lceil \omega(G)/3 \rceil .$$

For each isolated vertex $v \in G_2$ we should have $d_2(v) \leq \left\lceil \omega(G)/3 \right\rceil$.
Recall that $v$ is very heavy, so $d_2(v) = d(v) - 2\kappa(v)$ because the vertex has received $\kappa(v)$ colors in Step 1 and in Step 2.
We claim that  $d_2(v) \leq \kappa(v)$.
Indeed, if   $d_2(v) > \kappa(v)$, then $d(v) = d_2(v) + 2\kappa(v) > 3\kappa(v)$ contradicting the definition of $\kappa(v)$.
Hence, $d_2(v) \leq \kappa(v) \leq  \left\lceil \omega(G)/3 \right\rceil$ as needed.
\qed
\end{proof}
\\ \\

Let $\Delta(G)$ be the maximal vertex degree in the graph $G$. %, and $deg_G(v)$ represents the number of neighbors of vertex $v$ in the graph $G$. 
In Step 3,   we have to first  prove  that:

\begin{lem}\label{deg4}
$\Delta(G_2) \leq 4$ and every vertex $v$ with $deg_{G_2}(v) = 4$ has at least one free color.
\end{lem}

\begin{proof}
Let $v$ be an arbitrary vertex in the graph $G_2$. \ps{Recall that by Lemma \ref{G1tf} \ps{graph} $G_2$ is triangle-free. 
 Therefore, vertex $v$} can have at most 3 neighbors \ps{in its layer}, \ps{ and} the angle between any two of them is $2\pi / 3$. 
  In this case vertex $v$ cannot have any additional neighbor in \ps{the} lower or \ps{in the} upper layer, therefore $deg_{G_2}(v) = 3$. If the vertex $v$ has only 2 neighbors 
in its layer, then we have two different possibilities for the angle between the neighbors: $2\pi / 3$ and $\pi$. Both possibilities on \rw{the} layer-arrangement (a) are 
depicted on Figure \ref{Free_a} and the (b) case is depicted \ps{in} Figure \ref{Free_b}. It is easy to see  that vertex $v$ could have at most two additional neighbors in
 lower and upper layer - otherwise we obtain a triangle.
Suppose that $deg_{G_2}(v) = 4$, then all possible cases of its neighbourhood are shown \ps{in} Figures \ref{Free_a} and \ref{Free_b}. 
It is easy to see that in both cases (1) vertex $v$ can borrow color 1, and in both cases (2) vertex $v$ can borrow color 2 or 3.
% \raf{if we change figures, we should also make some changes here}
\end{proof}

\begin{figure}[h]
 \begin{center}
 \includegraphics[height=4cm]{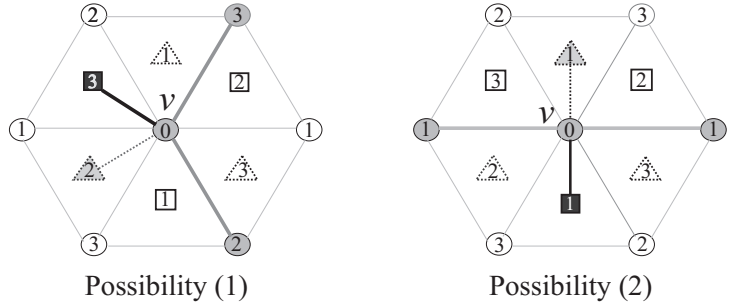}
 \caption{Two different possibilities for  neighbourhood of  vertex $v$ with $deg_{G_2}(v)=4$ in a triangle-free cannonball graph, obtained from \rw{the} layer-arrangement (a). Circles represent vertices of the middle layer, squares of the upper layer and triangles of the lower layer, and white vertices are  part of the grid, but \ps{are} not in the graph.
%Besides, colors of 4-coloring are represented as C (cyan), M (magenta), Y (yellow) and K (black).
 } 
 \label{Free_a}
 \end{center}
\end{figure} 

%\raf{Here, Janez suggest so that we can make so that color 1 is free color. Petra - you can change the numbers or you can change the position of the edges, into less horizontal}
%\pet{I woul'd leave the figure as it is, because we have the same vertex $v$ in the left and in the right figure and I think it will be confusing if I change the numbers,
% besides before Fig 3 we say that for the right case v can borrow colors 2 or 3 and in the last paragraph of page 10 we say 'Without loss of generality $bc(v)=1$... DO YOU AGREE?}
\begin{figure}[h]
 \begin{center}
 \includegraphics[height=4cm]{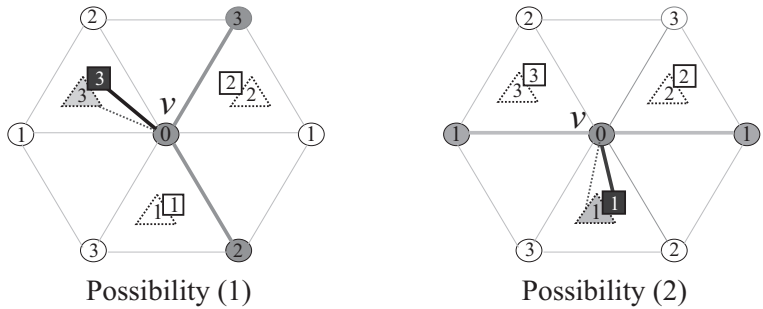}
 \caption{ Two different possibilities for   vertex $v$ neighbourhood with $deg_{G_2}(v)=4$ in a triangle-free cannonball graph, 
obtained from \rw{the} layer-arrangement (b). Notation has the same meaning as in Figure \ref{Free_a}.}
 \label{Free_b}
 \end{center}
\end{figure}

By Lemma \ref{deg4} we know that  borrowing is possible for all vertices of degree 4. 

In Step 3 we take the colors from the free base color palettes. Without loss of generality, assume that $bc(v) = 0$ and one of its free color is 1. Recall the function $D_1$ from page \pageref{D1def} -- we have $D_1(v)$ free colors from the first base color palette. We claim that

\begin{lem}\label{pozyczanie}
%If $v$ is a corner in $G_2$ with three slight neighbors colored blue, then
 $$d_2(v) \leq D_1(v)$$ 
\end{lem}

\begin{proof}
Let $v$ be a vertex in $G_2$ with $bc(v)=0$.
If   vertex $v \in G_2$ has \rw{ four} neighbors in $G_2$, it always has such free colors that all neighbors on its layer of this base colors are not in $V(G_2)$ 
(see Figures \ref{Free_a} and \ref{Free_b}). 
Without loss of generality assume that one of these free colors is 1.
Let $t$ be a vertex, which is an existing neighbor of $v$ in $G_2$, and $u$ is the neighbor of $v$ with $bc(u)=1$ so that $\{u,v,t\}\in\tau(T)$ is a~triangle.
Then we have
$$ \kappa(v) + d_2(v) + a(u,v,t) + d(u) \stackrel{\star}{\leq} d(v) + d(t) + d(u) \leq 3 a(u,v,t) \leq a(u,v,t) + 2\kappa(v)$$
%$$ d(u) + d_2(v) \leq \kappa(v) $$
%$$ d_2(v) \leq \kappa(v) - d(u) $$
and the inequality $\star$ occurs because $d_2(v) = d(v) - \kappa(v)$ and $d(t) > \kappa(t) \geq a(u,v,t)$.
Since vertex $u$ has to belong to \rw{ the} triangle $\{u,v,t\} \in \tau(T)$ where $t$ is a heavy vertex in $G$,
it holds $ d_2(v) \leq \kappa(v) - d(u)  $
and finally $ d_2(v) \leq  D_1(v).$
\end{proof}

In Step 4 we have to prove that for $G_3$ we can apply \ps{Procedure}   \ref{trzykolorowanie}. 
%\raf{Janez suggest this is not necessary to divide those two cases and he is probably right. Petra, do you agree?}
%\pet{I think he is right.}
%\rw{\sout{If $G_3$ is bipartite, we can surely use Procedure \ref{dwukolorowanie} by using $\omega(G_3)$ new colors.   If it is not, then ...........
We know that
$ \Delta(G_3) \leq 3$
since $ \Delta(G_2) \leq 4$ and in Step 3 we had fully colored all vertices with degree equal \rw{to} 4.
According to theorem   of Brooks  \cite{Brooks} we know that $G_3$ is 3-colorable. Therefore, we can apply Procedure \ref{trzykolorowanie} and multicolor  $G_3$  
 by using $\frac{3}{2}\omega(G_3)$ new colors.

\subsection*{Ratio}

We claim that during the first \ps{three steps} our algorithm uses at most $\frac{4}{3}\omega_3(G)+O(1)$ colors.
To see this, notice that in Step 1 each vertex $v$ uses at most $\kappa(v)$ colors from its base color palette and, 
by Fact \ref{komega} and using  that there are four base colors, we know that no more than 
$4\left\lceil \omega(G)/3 \right\rceil \leq \frac{4}{3}\omega(G)+O(1)$ colors are needed. Note also that in \ps{Step 2 and} Step 3 we use only 
those colors from \rw{the} base color palettes which were not used in Step 1, so altogether no more than $\frac{4}{3}\omega(G)+O(1)$ colors
 from \rw{the} base color palettes are used in total \jz{until} \ps{Step 4}.

\noindent In Step 4 we introduce  new palettes that contain no more than $\left\lceil \frac{3}{2}\omega(G_3) \right\rceil$ colors (by  Lemma \ref{kgoodlemma}) .

Let $A(G)$ denote the number of colors used by our algorithm for the graph $G$.
Thus, since $\omega(G_3) \leq \lceil \omega(G)/3 \rceil \leq \omega(G)/3 + O(1)$, the total number of colors used
by our algorithm is at most
$$ A(G) \leq \frac{4}{3}\omega(G) + \frac{3}{2}\omega(G_3) + O(1) = \frac{4}{3}\omega(G) + \frac{3}{6}\omega(G) + O(1) = \frac{11}{6}\omega(G) + O(1).$$

The performance ratio for our algorithm is $11/6$, hence we arrived at the statement of Theorem \ref{twierdzenie}.

%\section{Algorithm for case (b)}\label{caseb}\raf{maybe it is good idea to name it somehow}
%
%It is easy to see, that we can repeat Steps 1-4 from the previous algorithm, and all works fine
%
%\section{Algorithm for mix cases (a) and (b)}\label{casesab}\raf{maybe it is good idea to name it somehow}
%
%It is easy to see, that we can repeat Steps 1-4 from the previous algorithm, and all works fine.
%This is because, when we talk about neighborhood it is enough to see at one above and one below layer only. And using those three only, we can have situation from case (a) or (b), and no other.

%%%%%%%%%%%%%%%%%%%%%%%%%%%%%%%%%%%%%%%%%%%%%%%%%%%%%%%%
\section{Conclusion}
%%%%%%%%%%%%%%%%%%%%%%%%%%%%%%%%%%%%%%%%%%%%%%%%%%%%%%%%

In this paper we provide an algorithm for a proper multicoloring of \rw{a} cannonball graph that uses at most $\frac{11}{6}\omega(G)+C$ colors. 
As this is the first result for the multicoloring problem of  cannonball graphs, we belive that further improvements can be done.
Among the interesting problems that remain  open  are improvement of the competitive ratio $11/6$, finding some distributed algorithms for \ps{multicoloring} cannonball graphs,
or finding some $k$-local algorithms for some $k$, similarly as in 2D case for  hexagonal graphs  (for definition of $k$-local algorithms see \cite{Janssen}).
We already mentioned that in the 2D case, better bounds were obtained for  triangle-free hexagonal graphs. 
It is very likely that also for   cannonball graphs \jz{there} exist some "forbidden" subgraphs $H$, maybe tetrahedrons, 
such that better bounds can be obtained for $H$-free cannonball graphs.

\end{document}